\documentclass{amsart}

\usepackage{graphicx}

\newtheorem{theorem}{Theorem}[section]
\newtheorem{question}{Question}
\newtheorem{claim}{Claim}

\theoremstyle{remark}
\newtheorem{remark}{Remark}[section]

\begin{document}

\title{Hyperbolic small knots in spherical manifolds}

\author{Kazuhiro Ichihara}
\address{Department of Mathematics, College of Humanities and Sciences, Nihon University, 3-25-40 Sakurajosui, Setagaya-ku, Tokyo 156-8550, Japan}
\email{ichihara.kazuhiro@nihon-u.ac.jp}



\date{\today}

\thanks{This work was supported by JSPS KAKENHI Grant Number JP22K03301.}

\begin{abstract}
It was conjectured by Lopez that every closed irreducible non-Haken 3-manifold contains a small knot. 
In this paper, we give explicit examples of hyperbolic small knots in most closed orientable spherical 3-manifolds other than prism manifolds. 
\end{abstract}

\maketitle

\section{Introduction}

The following conjecture is supposed in \cite{Lopez92,Lopez}
that small knots always exist in irreducible non-Haken (not sufficiently large) 3-manifolds.
A knot is called \textit{small} if there are no closed embedded essential surfaces in its exterior.

Actually, Lopez states in \cite[Theorem A]{Lopez} that every Seifert fibered 3-manifold over the 2-sphere with at most three singular fibers contains infinitely many small hyperbolic knots. 
However, in \cite[remark after Conjecture 1.5, p.150]{Matsuda}, Matsuda claimed that his proof appears to be incomplete. 
Also, the Lopez's proof depends upon Thurston's hyperbolic Dehn surgery theorem, and so, no explicit examples were given. 

In this paper, we consider hyperbolic small knots in certain non-Haken Seifert fibered 3-manifolds; \textit{spherical} 3-manifolds, which are 3-manifolds endowed with the spherical geometric structure. 
It is known that those are actually Seifert fibered 3-manifold over the 2-sphere with at most three singular fibers. 
See \cite{Scott, AschenbrennerFriedlWilton} for example. 

The following is our main result; in the proof, we give explicit examples of hyperbolic small knots in most closed orientable spherical 3-manifolds other than prism manifolds. 

\begin{theorem}\label{thm:main}
Every spherical 3-manifold, except for prism manifolds and the Seifert fibered manifolds $\pm (-1; 1/2, 1/3, 1/m )$ with $m \in \{ 3, 4, 5 \}$, contains a hyperbolic small knot.
\end{theorem}

In fact, by the Geometrization theorem due to \cite{Perelman1,Perelman2}, a closed orientable 3-manifold is a spherical 3-manifold if and only if it has the finite fundamental group. 
It implies that it contains no essential surfaces, i.e., it is non-Haken. 
See \cite{Scott,AschenbrennerFriedlWilton} for example. 
We remark that non-hyperbolic small knots always exist in a non-Haken Seifert fibered 3-manifold with base space a sphere and three exceptional fibers: just take for the knot an exceptional fiber. 

In the next section, we briefly overview the history about (non-)Haken 3-manifolds and small knots. 
The proof of the main theorem will be given in Section 3.

\section{Backgrounds}

Please refer to \cite{AschenbrennerFriedlWilton, JacoBook} for standard definitions and notations that are not explained below.

Let $M$ be a compact, orientable 3-manifold.
A closed, orientable surface $F$ embedded in $M$ is called \textit{essential} if either $F$ is a 2-sphere that does not bound a 3-ball in $M$, or $F$ is incompressible and not parallel to any component of $\partial M$.
The manifold $M$ is called \textit{small} if it contains no closed essential surface.
A closed, irreducible 3-manifold $M$ is called \textit{Haken} if it contains a closed essential surface.
A knot $K$ in a closed, orientable 3-manifold $M$ is called \textit{small} if its exterior $E(K)$ is a small 3-manifold.
Here, the exterior $E(K)$ is the manifold obtained from $M$ by removing an open tubular neighborhood of $K$.

In \cite{Haken1961, Haken1962}, Haken introduced the concept of an \textit{incompressible surface} embedded in a 3-manifold.
Roughly speaking, incompressibility provides a geometric interpretation of the $\pi_1$-injectivity of embedded surfaces in 3-manifolds.
Today, a compact, orientable 3-manifold is called a \textit{Haken manifold} if it is irreducible and contains a properly embedded, two-sided incompressible surface.
Since then, it has been shown that Haken manifolds possess a variety of desirable properties.

In a sense, “most” 3-manifolds are Haken. 
For example, if $\mathrm{rank}\,H_1(M;\mathbb{Q}) \ge 1$ holds for a closed, orientable 3-manifold $M$, then $M$ is Haken.
Moreover, Agol \cite{Agol} proved the so-called Virtual Haken Conjecture, originally proposed by Waldhausen \cite{Waldhausen}.
This conjecture asserts that every compact, orientable, irreducible 3-manifold with infinite fundamental group is \textit{virtually Haken}, i.e., it admits a finite cover that is a Haken manifold.

On the other hand, it is known that not every closed, orientable 3-manifold is Haken.
In fact, Thurston demonstrated in his well-known lecture notes \cite{ThurstonLectureNote} that the figure-eight knot in the 3-sphere has no closed incompressible surface embedded in its exterior that is not boundary-parallel, and that all but finitely many 3-manifolds obtained by Dehn surgery on the figure-eight knot are non-Haken.
Extending this result, Hatcher showed in \cite{Hatcher} that all but finitely many 3-manifolds obtained by Dehn surgery on a small knot are non-Haken.
Here, a knot in a 3-manifold is called \textit{small} if its exterior is irreducible and every closed surface embedded in the exterior is either compressible or boundary-parallel.

This implies that, given a small knot, one can construct infinitely many non-Haken 3-manifolds.
In contrast, finding small knots in 3-manifolds is rather difficult in general. 
Only for knots in the 3-sphere $S^3$, all the torus knots are small; see \cite{JacoBook}, for example.
Actually, the exterior of a torus knot is a Seifert fibered space with two singular fibers over the disk, and such manifolds are known to contain no essential surfaces.
For other knots, it is known from \cite{HatcherThurston} and \cite{Oertel} that 2-bridge knots and Montesinos knots of length three are small. 
Also see \cite{BoyerZhang93,LinWei} for some closed braids and some alternating knots. 

Other than the 3-sphere, Lopez states in \cite[Theorem A]{Lopez} that every Seifert fibered 3-manifold over the 2-sphere with at most three singular fibers contains infinitely many small hyperbolic knots.
However, it was claimed by Matsuda that the proof contains a serious gap in \cite[remark after Conjecture 1.5, p.~150]{Matsuda}.

Nevertheless, the following conjecture raised in \cite{Lopez} can be considered valid: 
Every closed, orientable, irreducible, non-Haken 3-manifold contains a small knot.
This statement is now sometimes referred to as the \textit{Lopez Conjecture}.
See also \cite[Problem 1.13]{Rubinstein} for further discussion.

For some other constructions of small knots, see \cite{Matsuda} (in Sapphire spaces), \cite{Worden} (in $S^3$ with exteriors of large Heegaard genus), and \cite{QiuWang} (in handlebodies).

\section{Proofs}

It is known that a 3-manifold is spherical if and only if it is homeomorphic to the quotient of $S^3$ by a finite group acting freely and isometrically. 
See \cite{AschenbrennerFriedlWilton, Bonahon}, for example.
In particular, the fundamental group of a spherical 3-manifold can be viewed as a finite subgroup of $SO(4)$ that acts freely on $S^3$.
Then, in light of the classification of such subgroups, spherical 3-manifolds are classified into types $\mathbf{C}$, $\mathbf{D}$, $\mathbf{T}$, $\mathbf{O}$, or $\mathbf{I}$.
Here we omit the details; see \cite{AschenbrennerFriedlWilton, BoyerZhang}, for example. 

We only note the following:
A manifold of type $\mathbf{C}$ is a lens space $L(p,q)$ for coprime positive integers $p$ and $q$, and a manifold of type $\mathbf{D}$ is a prism manifold.
In this paper, the lens space $L(p,q)$ is assumed to be obtained by $(-p/q)$-surgery on the unknot in $S^3$. 

Then, Theorem~\ref{thm:main} follows from the following two theorems. 
In the proofs, we give explicit examples of hyperbolic small knots. 

\begin{theorem}\label{thm:lens}
Every lens space contains infinitely many hyperbolic small knot. 
\end{theorem}

Note that any spherical 3-manifold is a Seifert fibered space with at most 3 singular fibers over the sphere.
Moreover, any spherical 3-manifold of type $\mathbf{T}$, $\mathbf{O}$, or $\mathbf{I}$ admits a unique Seifert bundle structure.
See \cite{Scott}, for example. 

\begin{theorem}\label{thm:TOI}
Let $M$ be a spherical 3-manifold of type $\mathbf{T}$, $\mathbf{O}$, or $\mathbf{I}$. 
Suppose that $M$ is not a Seifert fibered manifold of the form $\pm (-1; 1/2, 1/3, 1/m)$ with $m \in { 3, 4, 5 }$.
Then, $M$ contains a hyperbolic small knot. 
\end{theorem}

\subsection{Lens spaces}

\begin{proof}[Proof of Theorem~\ref{thm:lens}]

Let $L(p,q)$ be the lens space for coprime integers $p > 0$ and $q \ne 0$.
We will find an infinite family of hyperbolic small knots in $L(p,q)$.

Consider the 2-bridge link $L_k = K \cup K'$ in $S^3$ corresponding to the continued fraction expansion $[2, 2k, -2]$ for an integer $k \ge 2$.
See Figure~\ref{fig:1}.
Refer to \cite{Ichihara2012, ichiharaMattman} for details.
Here we assume that $4k \ne \pm p/q$; there are infinitely many $k$ satisfying this condition.

 \begin{figure}[htb]
 \centering
    \includegraphics[width=.8\textwidth]{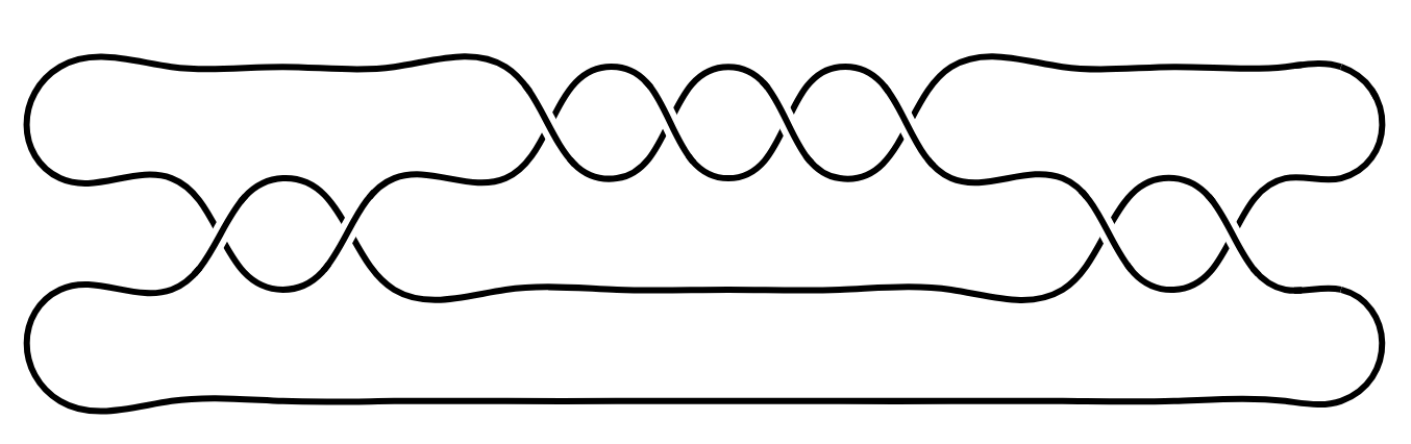}
 \caption{The link diagram $C(2,2k,-2)$ for $k=-2$.}\label{fig:1}
 \end{figure}

We perform Dehn surgery on $K' \subset L_k$ along the slope $-p/q$ to obtain the lens space $L(p,q)$, and regard $K$ as a knot in $L(p,q)$.

The following claim follows from the classification of exceptional surgeries on a component of a 2-bridge link, obtained in \cite[Theorem 1.1]{Ichihara2012}, and from the arguments in the proof of \cite[Lemma 3]{ichiharaMattman}. 

\begin{claim}
    The knot $K$ is a hyperbolic knot in $L(p,q)$. 
\end{claim}

\begin{proof}
It suffices to show that $(-p/q)$-surgery on $K'$ is not exceptional, which implies $K$ is hyperbolic.
Suppose to the contrary that $(-p/q)$-surgery on $K'$ is exceptional for $p > 0$ and $q \ne 0$.
We observe that the $0$-surgery on $K'$ is exceptional, indeed, it is toroidal, since $K'$ bounds a once-punctured torus in the exterior of $L_k$. 
Also it is shown in \cite[Proof of Lemma 3]{ichiharaMattman} that if a component of a hyperbolic 2-bridge link admits multiple exceptional surgeries, then there exists at least one Seifert surgery. 
Thus, $K'$ must admit a Seifert surgery. 
Then, by \cite[Theorem 1.1]{Ichihara2012}, $L_k$ is equivalent to the link represented by the continued fraction $[2w+1,2u+1]$ for $w \ge 1$, $u \ne 0, -1$.
Also see \cite[Lemma 3]{ichiharaMattman}.

By \cite[Theorem 2.1.11]{KawauchiBook} (see also \cite[Lemma 1]{ichiharaMattman}), to distinguish two-bridge links represented by continued fractions, it suffices to check their simple (i.e., all terms positive) continued fraction expansions.

In fact, since
\[
[2w+1,2u+1]
=
\begin{cases}
[2w+1,2u+1] & \text{if $w,u \ge 1$},\\
[2w,1,-2u-2] & \text{if $w \ge 1$ and $u \le -3$},
\end{cases}
\]
and $[2,2k,-2] = [2,2k-1,2]$ ($k \ge 2$) hold, we have a contradiction. 
\end{proof}

It remains to show that $K$ is a small knot.
Let $E(K)$ be the exterior of $K$ in $L(p,q)$.
Suppose that $K$ is not a small knot in $L(p,q)$, i.e., there exists a closed incompressible surface $F$ of genus at most two in $E(K)$. 
Then, we can find an essential surface $F'$ with boundary in $E(L_k)$ such that its boundary slope on $\partial N(K')$ is $-p/q$.
After performing meridional compressions on $F'$ for $K$, we obtain an essential surface in $E(L_k)$ with (possible) boundary on $K$ of slope $1/0$ and with boundary on $K'$ of slope $-p/q$.

Such boundary slopes for 2-bridge links are completely classified, and there is an algorithm to enumerate them.
See \cite{FloydHatcher, Lash, HosteShanahan}.
In particular, the list of boundary slope pairs for $L_k$ is given in \cite[Section 5, Table 4]{HosteShanahan}.
See Table~\ref{tab:1}.


\begin{table}[htb]
\centering
\caption{Boundary slopes for $L_k$ (both $t$ and $s$ are rational parameters).}
\label{tab:1}
\begin{tabular}{|l|c|}
\hline
$\partial$-slopes&restrictions\\
\hline
\hline
$(0,0)$&\\
$(0, \emptyset),(\emptyset, 0)$&\\
$(-4k, \emptyset),(\emptyset, -4k)$&\\
$(-4k, -2),(-2, -4k)$&\\
$(2 t^{-1}, 2t)$&$0\le t \le \infty$\\
$(-2 t^{-1}, -2t)$&$0\le t \le \infty, k>1$\\
$(-2 t^{-1}+2-4k, -2t)$&$0\le t \le 1$\\
$(-2 t^{-1},2-4k- 2t)$&$1\le t \le \infty$\\
$(-1-2k+(2k-1)s, -1-2k-(2k-1)s)$&$-1 \le s\le 1$\\
\hline
\end{tabular}
\end{table}

Since there are no slope pairs $\{ 1/0 \ (=\infty), -p/q\}$ or $\{ \emptyset, -p/q \}$ when $-p/q \ne 4k$, we obtain a contradiction. 
\end{proof}

\begin{remark}
The argument in the proof above is also applicable to the manifold $S^2 \times S^1$.
That is, there exist infinitely many hyperbolic small knots in $S^2 \times S^1$.
\end{remark}

\begin{remark}
Hyperbolic small knots in lens spaces can also be found as hyperbolic genus one fibered knots, since such knots are small by \cite{FloydHatcher1982, CullerJacoRubinstein}.
A complete list of lens spaces containing genus one fibered knots is provided in \cite{Baker}, which shows that some lens spaces do not contain any such fibered knots. 
\end{remark}

\subsection{Spherical manifolds of type $\mathbf{T}$, $\mathbf{O}$, $\mathbf{I}$}

\begin{proof}[Proof of Theorem~\ref{thm:TOI}]
It is known that a spherical 3-manifold is either a manifold of type $\mathbf{T}$, $\mathbf{O}$, or $\mathbf{I}$, which is a Seifert 3-manifold described as $\pm (-1 ; \frac{1}{2}, \frac{1}{3}, \frac{a_3}{b_3})$ with $b_3 \in { 3,4,5 }$ and $(a_3,b_3)=1$.
See \cite[Theorem 6]{Doig}, for example.
Moreover, it is obtained by $(6 - b_3/a_3)$-surgery on the right-hand trefoil $T_{3,2}$ in $S^3$.
See \cite[Corollary 8]{Doig}, for example.

We here consider the Whitehead link $L = K \cup K'$, which is represented by the continued fraction $[2,2,-2]$.
(The link $L_k$ in the previous section with $k=1$.)
We note that the right-hand trefoil is obtained by $1$-surgery on $K' \subset L_1$.
It follows that a spherical 3-manifold of type $\mathbf{T}$, $\mathbf{O}$, or $\mathbf{I}$ is obtained by Dehn surgery on $L = K \cup K'$ along the pair of slopes $(6 - b_3/a_3, 1)$.

Now, let $M$ be a spherical 3-manifold of type $\mathbf{T}$, $\mathbf{O}$, or $\mathbf{I}$, which is not a Seifert fibered manifold $\pm (-1; 1/2, 1/3, 1/m )$ with $m \in { 3, 4, 5 }$.
Then, by the previous arguments, $M$ is obtained by Dehn surgery on $L = K \cup K'$ along the pair of slopes $(6 - b_3/a_3, 1)$ with $b_3 \in { 3,4,5 }$, $(a_3,b_3)=1$ and $a_3 \ne 1$.

We consider the dual knot $K''$ of $K'$, that is the core of the solid torus attached to $\partial N(K')$.
This $K''$ is a hyperbolic knot in $M$, since $(6 - b_3/a_3)$-surgery on $K \subset L$ is not exceptional.
See \cite[Theorem 3.1]{MartelliPetronioRoukema} and \cite{Ichihara2012}. 

It remains to show that this $K''$ is a small knot in $M$.
In the same way as the proof of Theorem~\ref{thm:lens}, if it is not, there would exist an essential surface in $E(L)$ with (possibly) boundary on $K$ of slope $1$ and with boundary on $K'$ of slope $6 - b_3/a_3$.
However, from Table~\ref{tab:1}, we have a contradiction.
\end{proof}

Unfortunately, this proof does not apply to the case where $a_3 = 1$, since $1$-, $2$-, and $3$-surgeries on $L$ are actually exceptional.

\section{Question}

For manifolds with finite fundamental groups, the remaining cases are Prism manifolds and Seifert fibered spaces of the form $\pm (-1; 1/2, 1/3, 1/m)$ with $m \in { 3, 4, 5 }$.

For coprime integers $n, m$, a Seifert fibered manifold presented by $(-1; 1/2, 1/2, m/n)$ with $S^2$ as the base space is called a \textit{Prism manifold} $P(n,m)$. 
This can also be expressed as a twisted $I$-bundle over the Klein bottle, glued with a solid torus such that the pair $(n, m)$ on the boundary of the $I$-bundle bounds a disk in the attached solid torus. 

Thus, the following is a natural question:

\begin{question}
Can we find explicit examples of hyperbolic small knots in a prism manifold? 
\end{question}

Our argument in this paper cannot apply to prism manifolds, since they are not obtained by Dehn surgeries on 2-bridge links. 

\section*{Acknowledgments}
The author would like to thank Luis Miguel Lopez and Hiroshi Matsuda for helpful conversation about their paper \cite{Lopez,Matsuda}.

\bibliographystyle{amsplain}


\providecommand{\bysame}{\leavevmode\hbox to3em{\hrulefill}\thinspace}
\providecommand{\MR}{\relax\ifhmode\unskip\space\fi MR }
\providecommand{\MRhref}[2]{%
  \href{http://www.ams.org/mathscinet-getitem?mr=#1}{#2}
}
\providecommand{\href}[2]{#2}

\end{document}